\documentclass[11pt]{article}
\usepackage{a4wide}

\usepackage[a4paper, margin=2.3cm]{geometry}
\usepackage{amsmath,amssymb,amsthm}
\usepackage{color}
\usepackage{parskip}
\usepackage{hyperref}
\usepackage{parskip}
\usepackage{setspace}
\usepackage{graphicx}

\newtheorem{theorem}{Theorem}[section]

\newtheorem{corollary}[theorem]{Corollary}

\newtheorem{lemma}[theorem]{Lemma}

\newtheorem{proposition}[theorem]{Proposition}
\newtheorem{definition}[theorem]{Definition}

\newtheorem*{definition*}{Definition}

\begin{document}
\title{Structural theorems on the distance sets over finite fields}
\author{Doowon Koh\thanks{Department of Mathematics, Chungbuk National University, Korea. Email: koh131@chungbuk.ac.kr}\and Minh Quy Pham\thanks{Department of Mathematics, Da Lat University. Email: p.minhquydl@gmail.com}\and Thang Pham\thanks{University of Science, Vietnam National University, Hanoi, Vietnam. Email: thangpham.math@vnu.edu.vn}}

\maketitle
\begin{abstract}
Let $\mathbb{F}_q$ be a finite field of order $q$. 
Iosevich and Rudnev (2007) proved that for any set $A\subset \mathbb{F}_q^d$, if $|A|\gg q^{\frac{d+1}{2}}$, then the distance set $\Delta(A)$ contains a positive proportion of all distances. Although this result is sharp in odd dimensions, it is conjectured that the right exponent should be $\frac{d}{2}$ in even dimensions. During the last 15 years, only some improvements have been made in two dimensions, and the conjecture is still wide open in higher dimensions. To fill the gap, we need to understand more about the structures of the distance sets, the main purpose of this paper is to provide some structural theorems on the distribution of square and non-square distances.  
\end{abstract}
\section{Introduction}

Let $\mathbb{F}_q$ be a finite field of order $q$ which is a prime power. The distance between two points $x$ and $y$ in the space $\mathbb{F}_q^d$ is defined by \[||x-y||:=(x_1-y_1)^2+\cdots+(x_d-y_d)^2.\]
For $A\subset \mathbb{F}_q^d$, let $\Delta(A)$ denote the set of distances determined by pairs of points in $A$, i.e. 
\[\Delta(A)=\{||x-y||\colon x, y\in A\}.\]
Studying the magnitude of $\Delta(E)$ has been received much attention since 2007 due to its connection to the Falconer distance conjecture in Geometric Measure Theory, which says that for any compact set $A\subset \mathbb{R}^d$ of Hausdorff dimension greater than $d/2$, the distance set $\Delta(A)$ is of positive Lebesgue measure. Recent progress on this conjecture can be found in \cite{Du1, Du3, Du2, alex-fal}.

%The followings are the best current dimensional thresholds for this conjecture: 
%\begin{itemize}
%\item $d=2$, Guth, Iosevich, Ou, and Wang \cite{alex-fal} (2019): $1+\frac{1}{4}$
%\item $d=3$, Du, Guth, Ou, Wang, Wilson, and Zhang \cite{Du1} (2017): $\frac{3}{2}+\frac{3}{10}$
%\item $d\ge 4$ even, Du, Iosevich, Ou, Wang, and Zhang \cite{Du3} (2020): $\frac{d}{2}+\frac{1}{4}$
%\item $d\ge 5$ odd, Du and Zhang \cite{Du2} (2018): $\frac{d}{2}+\frac{d}{4d-2}$
%\end{itemize}
In the finite field setting, Iosevich and Rudnev \cite{IR07} proved that for $A\subset \mathbb{F}_q^d$, if $|A|\ge 4q^{\frac{d+1}{2}}$, then $\Delta(A)=\mathbb{F}_q$. The exponent $\frac{d+1}{2}$ is sharp in odd dimensions, namely, $d\equiv 3\mod 4$ and $q\equiv 1\mod 4,$ or $d\equiv 1\mod 4$, we refer the reader to \cite{HIKR10} for constructions. In even dimensions, it is conjectured that the right exponent should be $\frac{d}{2}$ which is directly in line with the Falconer distance conjecture. In a very recent paper, Murphy, Petridis, Pham, Rudnev, and Stevens \cite{MPT} obtained the exponent $\frac{5}{4}$ in the plane over prime fields, which improves the exponent $\frac{4}{3}$ by Chapman, Erdogan, Hart, Iosevich, and Koh in \cite{CEHIK10}. In particular, they showed that for $A\subset \mathbb{F}_p^2$ with $|A|\gg p^{5/4}$, then there exists $x\in A$ such that
\begin{equation}\label{54}|\Delta_x(A)|\gg p,\end{equation}
where $\Delta_x(A):=\{||x-a||\colon a\in A\}$. This matches the breakthrough dimensional threshold $\frac{5}{4}$ on the Falconer distance conjecture given by Guth, Iosevich, Ou, and Wang in \cite{alex-fal}. For $4p<|A|\le p^{5/4}$, Murphy et al. also proved that there exists $x\in A$ with
\begin{equation}\label{dim}
    |\Delta_x(A)|\gg \frac{|A|^{4/3}}{p^{2/3}}.
\end{equation}
This statement is similar to a result of Liu in \cite{liu} for the continuous setting which states that for any compact set $A\subset \mathbb{R}^2$, if the Hausdorff dimension of $A$ is greater than $1$, then 
\[\dim_H(\Delta_x(A))\ge \min \left\{\frac{4}{3}\dim_H(A)-\frac{2}{3}, ~1 \right\},\]
for some $x\in A$. We note that when the dimension of $A$ is very close to $1$, then Shmerkin \cite{hai} obtained better results, namely, 
\[\dim_H(\Delta_x(A))\ge \frac{29}{42},\]
and 
\[\dim_H(\Delta(A))\ge \frac{40}{57}.\]
Over finite fields, in higher even dimensions or $d\equiv 3\mod 4$ and $q\equiv 3\mod 4$, the best current exponent is still $\frac{d+1}{2}$ due to Iosevich and Rudnev \cite{IR07}, which is far from the conjecture $\frac{d}{2}$. In order make further progress, we need to understand more about structures of the distance sets, and in this paper, we focus on the distribution of pairs of points of square and non-square distances in a given set $A\subset \mathbb{F}_q^d$.

To state our main results, we need the following definition.
\begin{definition} Let $\eta$ denote the quadratic character of $\mathbb F_q$ with $\eta(0)=0.$ For $A\subset \mathbb F_q^d,$ we define 
$$  \mathcal{SQ}(A):=|\{(x, y)\in A\times A: \eta(||x-y||)=1\}|,$$
and 
$$ \mathcal{ZR}(A):=|\{(x, y)\in A\times A: \eta(||x-y||)=0\}|,$$

%and
%$$A^{-1}:=|\{(x, y)\in A\times A: \eta(||x-y||)=-1\}|.$$
as the sets of pairs of square distances and zero-distance, respectively. 
\end{definition}

For a given $A\subset \mathbb{F}_q^d$, if we consider the zero distance as a square, then either the number of pairs of square distances or the number of pairs of non-square distances is at least $|A|^2/2$. 

In the first theorem, we show that when the size of $A$ is not small, most of the pairs of points in $A$ are of non-square distances. Precisely, we have  
\newpage
\begin{theorem}\label{mainN0} Let $A\subset \mathbb F_q^d.$ 
 \begin{enumerate}
 \item If $d\equiv 3 \mod{4}$ and $ q\equiv 3 \mod{4},$ then
 $$\mathcal{SQ}(A)+\mathcal{ZR}(A) \le \frac{|A|^2}{2} +\frac{|A|^2}{2q} 
 - \frac{|A|^2}{2 q^{\frac{d-1}{2}}} - \frac{|A|^2}{2 q^{\frac{d+1}{2}}}+ q^{\frac{d-1}{2}} |A|.$$
 \item If $d\equiv 1 \mod{4},$ or $d\equiv 3 \mod{4}$ and $q\equiv 1 \mod{4},$ then
 $$\mathcal{SQ}(A)+\mathcal{ZR}(A) \le \frac{|A|^2}{2} +\frac{|A|^2}{2q}+ \frac{q^{\frac{d+1}{2}} |A|}{2} -\frac{q^{\frac{d-1}{2}} |A|}{2}.$$
 \item If $d\equiv 2 \mod{4}$ and $ q\equiv 3 \mod{4},$ then
 $$\mathcal{SQ}(A)+\mathcal{ZR}(A) \le \frac{|A|^2}{2} +\frac{|A|^2}{2q}-\frac{|A|^2}{q^{\frac{d}{2}}}+ \frac{q^{\frac{d}{2}} |A|}{2} +\frac{q^{\frac{d-2}{2}}|A|}{2}.$$
 \item If $d\equiv 0 \mod{4},$ or $d\equiv 2 \mod{4}$ and $q\equiv 1 \mod{4},$ then
 $$ \mathcal{SQ}(A)+\mathcal{ZR}(A) \le \frac{|A|^2}{2} +\frac{|A|^2}{2q}+ \frac{q^{\frac{d}{2}} |A|}{2}-\frac{q^{\frac{d-2}{2}}|A|}{2}.$$
 \end{enumerate}
 \end{theorem}
 
We say that a set $A\subset \mathbb F_q^d$ is a \textit{square distance} set if for all $x,y\in A, ~||x-y||$ is either zero or a square number of $\mathbb F_q.$ It is clear that for such a set, one has $\mathcal{SQ}(A)+ \mathcal{ZR}(A)=|A|^2$ for any square distance set $A$ in $\mathbb F_q^d.$ 
Hence, if $A\subset \mathbb F_q^d$ is a square distance set, then $|A|^2$ is dominated by the upper bound of $\mathcal{SQ}(A)+ \mathcal{ZR}(A)$ given in Theorem \ref{mainN0}. By solving the inequalities in terms of $|A|$,  we recover the following results of Iosevich, Shparlinski, and Xiong \cite{iosevich}. We also refer the reader to \cite{i1, i2, i3, iosevich} for more discussions and the motivation of this problem.
 \begin{corollary}[Iosevich-Shparlinski-Xiong, \cite{iosevich}]\label{cor:1.3}
 Let $A$  be a square distance set in $\mathbb F_q^d.$
 \begin{enumerate}
 \item If $d\equiv 3 \mod{4}$ and $ q\equiv 3 \mod{4},$ then
 $$|A| \le \frac{2 q^{\frac{d+1}{2}}}{q-1+ (q+1) q^{-\frac{(d-1)}{2}}}.$$ 

 \item If $d\equiv 1 \mod{4},$ or $d\equiv 3 \mod{4}$ and $q\equiv 1 \mod{4},$ then
 $$ |A|\le q^{\frac{d+1}{2}}.$$
 \item If $d\equiv 2 \mod{4}$ and $ q\equiv 3 \mod{4},$ then
 $$|A| \le q^{\frac{d}{2}}+ \frac{2 (q^{\frac{d}{2}}-q)}{q-1+ 2 q^{-\frac{(d-2)}{2}}}.$$ 

 \item If $d\equiv 0 \mod{4},$ or $d\equiv 2 \mod{4}$ and $q\equiv 1 \mod{4},$ then
  $$ |A|\le q^{\frac{d}{2}}.$$
 \end{enumerate}
 \end{corollary}

In the next two theorems, we are interested in the distribution of non-zero square distances in a given set $A\subset \mathbb{F}_q^d$.
%upper bounds for $\mathcal{SQ}(A)$ and $\mathcal{ZR}(A)\cup \mathcal{SQ}(A)$. 
\begin{theorem}[odd dimensions]\label{mainNOdd} Let $A$ be a subset of $\mathbb F_q^d.$ 
 \begin{enumerate} 
\item Let $d\equiv 3 \mod{4}$ and $ q\equiv 3 \mod{4}.$ 

 If $|A|\ge (q^{(d+1)/2} +q)/( 1+ q^{-(d-1)/2}),$ then 
$$ \mathcal{SQ}(A)\le \frac{|A|^2}{2} + q^{\frac{d-1}{2}} |A| -\frac{|A|^2}{2q}- \frac{|A|^2}{2 q^{\frac{d-1}{2}}} - \frac{|A|^2}{2 q^{\frac{d+1}{2}}}.$$
If $|A|\le (q^{(d+1)/2} +q)/( 1+ q^{-(d-1)/2}),$ then 
$$ \mathcal{SQ}(A) \le \frac{|A|^2}{2} + \frac{q^{\frac{d-1}{2}} |A|}{2}
  - \frac{|A|^2}{2 q^{\frac{d-1}{2}}} -\frac{|A|}{2}.$$

  \item  Let $d\equiv 1 \mod{4},$ or $d\equiv 3 \mod{4}$ and $q\equiv 1 \mod{4},$ then 
 $$ \mathcal{SQ}(A) \le \frac{|A|^2}{2} - \frac{q^{\frac{d-1}{2}} |A|}{2}
  -\frac{|A|}{2}+ \min\left\{ \frac{q^{\frac{d+1}{2}} |A|}{2},~ \frac{q^{\frac{d-1}{2}} |A|}{2}+ \frac{|A|^2}{2},~\frac{|A|}{2}+ \frac{q^{\frac{d+1}{2}} |A|}{2}-\frac{|A|^2}{2q} \right\}.$$
 \end{enumerate}
 \end{theorem}
 \begin{theorem}[even dimensions] \label{mainNEven} Let $A$ be a subset of $\mathbb F_q^d.$ 
 
 \begin{enumerate}
 \item
 If  $d\equiv 2 \mod{4}$ and $ q\equiv 3 \mod{4},$ then
 %If $|A|\ge (q^{d/2} +q)/( 1- q^{-(d-2)/2}),$ then 
$$ \mathcal{SQ}(A)\le \frac{|A|^2}{2} + \frac{q^{\frac{d}{2}} |A|}{2} -\frac{|A|^2}{2q}- \frac{q^{\frac{d-2}{2}}|A|}{2}.$$
%If $|A|\le (q^{d/2} +q)/( 1- q^{-(d-2)/2}),$ then 
%$$ \mathcal{SQ}(A) \le \frac{|A|^2}{2} + \frac{q^{\frac{d}{2}} |A|}{2}  - \frac{|A|^2}{2 q^{\frac{d}{2}}} -\frac{|A|}{2}.$$
  
\item Let $d\equiv 0 \mod{4},$ or $d\equiv 2 \mod{4}$ and $q\equiv 1 \mod{4}.$ 

If $|A|\ge (q^{d/2} +q)/( 1+ q^{-(d-2)/2}),$ then 
$$ \mathcal{SQ}(A)\le \frac{|A|^2}{2} + \frac{q^{\frac{d}{2}} |A|}{2}-\frac{|A|^2}{q^{\frac{d}{2}}} -\frac{|A|^2}{2q}+ \frac{q^{\frac{d-2}{2}}|A|}{2}.$$
If $|A|\le (q^{d/2} +q)/( 1+q^{-(d-2)/2}),$ then
$$ \mathcal{SQ}(A) \le \frac{|A|^2}{2} + \frac{q^{\frac{d}{2}} |A|}{2}
  - \frac{|A|^2}{2 q^{\frac{d}{2}}} -\frac{|A|}{2}.$$
\end{enumerate}
\end{theorem}
It has been proved in \cite{iosevich} that Corollary \ref{cor:1.3} (2) and (4) are sharp, so the upper bounds of Theorem \ref{mainN0} (2) and (4) are also best possible. It follows from Corollary \ref{cor:1.3} that if the set of non-zero distances is fully contained in the group of squares, then the size of $A$ can not be bigger than a certain threshold. After posting this paper to Arxiv, Prof. Igor Shparlinski raised the question of studying the case for an arbitrary subgroup of $\mathbb{F}_q\setminus \{0\}$. We hope to address this question in a sequel paper. 

%Our next theorem is on upper bounds of $\mathcal{SQ}(A) + \mathcal{ZR}(A)$.

%We have the following result.
%\begin{theorem}\label{main1} Let $A\subset \mathbb F_q^d.$ 
%\begin{enumerate}
%\item If $d, q\equiv 3 \mod{4},$ then  $\mathcal{SQ}(A)\le \frac{|A|^2}{2} + \frac{q^{(d-1)/2}|A|}{2}.$
%\item If $d$ is even, then $\mathcal{SQ}(A)\le \frac{|A|^2}{2} + \frac{q^{d/2}|A|}{2}.$
%\end{enumerate}
%\end{theorem}

\section{Discrete Fourier analysis and preliminary lemmas}
In this section, we recall notations from Discrete Fourier analysis and properties. 
Let $f$ be a complex valued function on $\mathbb F_q^n,$ then its Fourier transform denoted by $\widehat{f}$ is defined by
$$ \widehat{f}(m)=q^{-n} \sum_{x\in \mathbb F_q^n} \chi(-m\cdot x) f(x),$$ 
here $\chi$ is the principal additive character of $\mathbb F_q.$ 

With this definition, the Fourier inversion theorem reads as
$$ f(x)=\sum_{m\in \mathbb F_q^n} \chi(m\cdot x) \widehat{f}(m),$$
we also have the Plancherel theorem
$$ \sum_{m\in \mathbb F_q^n} |\widehat{f}(m)|^2 =q^{-n}\sum_{x\in \mathbb F_q^n} |f(x)|^2,$$
which can be proved easily by using the following orthogonal property
$$ \sum_{\alpha\in \mathbb F_q^n} \chi(\beta\cdot \alpha)
=\left\{\begin{array}{ll} 0\quad &\mbox{if}\quad \beta\ne (0,\ldots, 0),\\
q^n\quad &\mbox{if}\quad \beta=(0,\ldots,0). \end{array}\right.$$
%For example, it follows from the Plancherel theorem  that for any set $E$ in $\mathbb F_q^n,$ 
%$$ \sum_{m\in \mathbb F_q^n} |\widehat{E}(m)|^2= q^{-n}|E|.$$
%Throughout this paper,  we identify a set $E$ with the indicator function $1_E$ on $E.$\\

In this paper, the quadratic character of $\mathbb F_q$ is denoted by $\eta$ with a convention that $\eta(0)=0.$ For $a\in \mathbb F_q^*,$ the Gauss sum $G_a$ is defined by
$$G_a=\sum_{s\in \mathbb F_q^*}\eta(s) \chi(as),$$
which can be rewritten as
$$ G_a=\sum_{s\in \mathbb F_q} \chi(as^2)=\eta(a) G_1.$$
It is not hard to see that $|G_a|=\sqrt{q}$. In the following lemma, we recall the explicit form of the Gauss sum $G_1$ from \cite[Theorem 5.15]{LN97}.
\begin{lemma}\label{ExplicitGauss}
Let $\mathbb F_q$ be a finite field with $ q= p^{\ell},$ where $p$ is an odd prime and $\ell \in {\mathbb N}.$
Then we have
$$G_1= \left\{\begin{array}{ll}  {(-1)}^{\ell-1} q^{\frac{1}{2}} \quad &\mbox{if} \quad p \equiv 1 \mod 4 \\
{(-1)}^{\ell-1} i^\ell q^{\frac{1}{2}} \quad &\mbox{if} \quad p\equiv 3 \mod 4.\end{array}\right. $$
\end{lemma}

%As a direct corollary of the above lemma, we have the following result, which will be used repeatedly in proving our main results.

\begin{corollary} \label{Corm}Let $n$ be a positive integer. 
\begin{enumerate}
\item
If $n\equiv 0 \mod {4}$ and $q\equiv 3 \mod{4},$ then 
$\eta(-1)G_1^n= -q^{\frac{n}{2}}.$
\item If $n\equiv 2 \mod{4},$ or $n\equiv 0 \mod{4}$ and $q\equiv 1 \mod{4}$, then 
$\eta(-1)G_1^n= q^{\frac{n}{2}}.$
\item If $n\equiv 2 \mod{4}$ and $q\equiv 3 \mod{4}$, then 
$ G_1^n= -q^{\frac{n}{2}}.$
\item
If $n\equiv 0 \mod {4},$ or $n\equiv 2 \mod{4}$ and $q\equiv 1 \mod{4},$ then 
$G_1^n= q^{\frac{n}{2}}.$
\end{enumerate}
\end{corollary} 

The following formula, which can be proved by completing the square and using a change of variables, will be used in our computations below
\begin{equation}\label{ComSqu}  
 \sum_{s\in \mathbb F_q} \chi(as^2+bs)= \eta(a)G_1 \chi\left(\frac{b^2}{-4a}\right),\end{equation}
where $a\in \mathbb F_q^*$ and $b\in \mathbb F_q.$

Let $P\in \mathbb F_q[x_1,\ldots,x_n]$ be a polynomial, we define a variety $V_P$ as
$$V_P:=\{x\in \mathbb F_q^n: P(x)=0\}.$$
In the next lemma, we count the number of pairs of points $(x, y)$ in a given set in $\mathbb{F}_q^n$ such that $x-y\in V_P$. 

%We will often use this formula, which will be named the complete square formula.
%The discrete Fourier analysis is very useful in estimating the cardinality  of a specific set satisfying certain conditions.
%For instance, we have the following formula.
\begin{lemma}\label{AlexF} Let $E$ be a set in $\mathbb{F}_q^n$. Then the number of pairs $(x,y)\in E\times E$ such that $P(x-y)=0$ is equal to 
$$ q^{2n}\sum_{m\in \mathbb F_q^{n}} \widehat{V_P}(m)|\widehat{E}(m)|^2.$$
\end{lemma}
\begin{proof}
Let $N$ be the number of pairs $(x,y)$ such that $P(x-y)=0$, i.e.
$$ N=\sum_{x,y\in E} V_P(x-y).$$
By identifying the variety $V_P$ with the indicator function $1_{V_P}$ and using the Fourier inversion formula for the function $V_P(x-y)$ one has
$$ N=\sum_{x,y\in \mathbb F_q^n} E(x)E(y) \sum_{m\in \mathbb F_q^n} \widehat{V_P}(m) \chi(m\cdot (x-y)).$$ 
Then the lemma follows directly from the orthogonal property. 
\end{proof}

As we shall see  in the proofs of our main results, given $A\subset \mathbb F_q^d,$ we shall relate the value $\mathcal{SQ}(A) + \mathcal{ZR}(A)$ to the Fourier decay of the cone in $\mathbb F_q^{d+1}.$
For a positive integer $n\ge 2,$ we recall that the cone $C_n$ is defined as
\begin{equation}\label{conDe}C_n:=\{x\in \mathbb F_q^n: x_n^2=x_1^2+x_2^2+\cdots+ x_{n-1}^2\}.\end{equation}

\begin{definition}
Let $x=(x_1,x_2,\ldots, x_n) \in \mathbb F_q^n.$ We define 
$$ ||x||_{C_n}= x_1^2+x_2^2+\cdots+ x_{n-1}^2 - x_n^2.$$
\end{definition}
Using the notation $|| \cdot ||_{C_n}$,  the cone $C_n$ can be written by
$$ C_n=\{x\in \mathbb F_q^n: ||x||_{C_n}=0\}.$$
The next lemma contains the Fourier transform formula for $C_n$ which was computed explicitly in \cite{IMRN, KPP}.
%The Fourier transform of $C_n$ has been computed explicitly in \cite{IMRN, KPP}. For the reader convenience, we reproduce a proof here. 
\begin{lemma}\label{FTForm} Let $C_n$ be the cone in $\mathbb F_q^n$ defined in \eqref{conDe}.
Then, for any $m\in \mathbb F_q^n,$ we have
\begin{equation}\label{FFC}\widehat{C_n}(m)=q^{-1}\delta_{0}(m) +q^{-n-1} \eta(-1) G_1^{n} \sum_{s\ne 0} \eta^n(s) \chi\left( \frac{||m||_{C_n}}{-4s}\right).\end{equation}
%In particular, we have the followings:
%\begin{enumerate}
%\item
%If $n=4k$ for some $k\in \mathbb N,$ and $q\equiv 3 \mod{4},$ then we have
% $$ \widehat{C_n}(m)=\left\{ \begin{array}{ll} q^{-1}\delta_0(m)- q^{-\frac{n}{2}}+q^{-\frac{(n+2)}{2}}  \quad &\mbox{if}~~ m\in C_n\\
% q^{-\frac{(n+2)}{2}}   \quad &\mbox{if} ~~ m\notin C_n. \end{array}\right.$$
%
%\item If $n=4k+2$ for some $k\in \mathbb N$, or $n=4k$ for some  $k\in \mathbb N$ and $q\equiv 1 \mod{4},$ then we have
%$$ \widehat{C_n}(m)=\left\{ \begin{array}{ll} q^{-1}\delta_0(m)+ q^{-\frac{n}{2}}-q^{-\frac{(n+2)}{2}}  \quad &\mbox{if}~~ m\in C_n\\
% -q^{-\frac{(n+2)}{2}}   \quad &\mbox{if} ~~ m\notin C_n. \end{array}\right.$$
%
%
%\item
%If $n\ge 3$ is odd, then we have
%$$ \widehat{C_n}(m)= q^{-1}\delta_{0}(m) + q^{-n-1} \eta(\|m\|_{C_n}) G_1^{n+1},$$
%where we use the convention that $\eta(0)=0.$
% \end{enumerate}
 \end{lemma}

 \section{Preliminary settings on estimates of $\mathcal{SQ}(A)$ and $\mathcal{SQ}(A)+ \mathcal{ZR}(A)$} 
 For efficient estimates, we decompose $\mathcal{SQ}(A)$ as the sum of $\mathcal{SQ}(A) + \mathcal{ZR}(A)/2$ and $-\mathcal{ZR}(A)/2,$ and we also consider
 $\mathcal{SQ}(A)+ \mathcal{ZR}(A)$ as the sum of $\mathcal{SQ}(A) + \mathcal{ZR}(A)/2$ and $\mathcal{ZR}(A)/2.$ Therefore, we begin by doing delicate estimates on both  $\mathcal{SQ}(A) + \mathcal{ZR}(A)/2$ and $\mathcal{ZR}(A)/2.$ 
 
 \subsection{$\mathcal{SQ}(A)+ \mathcal{ZR}(A)/2$ value}
 Let $A\subset \mathbb F_q^d.$ %To prove Theorem  We aim to find upper bounds of $\mathcal{SQ}(A)$ and $
 Recall that $\mathcal{ZR}(A)$ and $\mathcal{SQ}(A)$ denote the numbers of pairs $(x, y)\in A\times A$ such that $||x- y||$ is zero and is a square number in $\mathbb F_q^*$, respectively; namely $||x- y||=0$ and $||x-y||=r^2$ for some $r\in \mathbb F_q^*.$ Since $ r^2=(-r)^2$ for $r\in \mathbb F_q^*,$  we can write
 $$ \mathcal{SQ}(A)= \frac{1}{2} \sum_{r\in \mathbb F_q^*} \sum_{x, y\in A: ||x-y||=r^2} 1.$$
 Since $\mathcal{ZR}(A)=\sum\limits_{x,y\in A: ||x-y||=0} 1$,  we have
 $$ \mathcal{SQ}(A)+ \frac{\mathcal{ZR}(A)}{2} =\frac{1}{2} \sum_{r\in \mathbb F_q} \sum_{x, y\in A: ||x-y||=r^2} 1.$$

 Now notice that for each fixed $r\in \mathbb F_q,$ there are exactly $q$ pairs $(s,s')\in \mathbb F_q\times \mathbb F_q$ such that 
 $r=s-s'.$
 This implies that
 \begin{align*} \mathcal{SQ}(A)+ \frac{\mathcal{ZR}(A)}{2}&=  \frac{1}{2q} \sum_{s,s'\in \mathbb F_q}
 \sum_{x, y\in A: ||x-y||=(s-s')^2} 1\\
 &= \frac{1}{2q} \sum_{\substack{(x,s), (y, s')\in A\times \mathbb F_q:\\
 ||(x,s)- (y, s')||_{C_{d+1}}=0}} 1. \end{align*}
 Put $E=A\times \mathbb F_q \subset \mathbb F_q^{d+1}$ and $ \widetilde{x}=(x,s), \widetilde{y}=(y, s')\in \mathbb F_q^{d+1}.$ It follows 
 $$\mathcal{SQ}(A)+ \frac{\mathcal{ZR}(A)}{2}= \frac{1}{2q} \sum_{\widetilde{x},\widetilde{y}\in E:\widetilde{x}-\widetilde{y}\in C_{d+1}}1.$$
 We now use Lemma \ref{AlexF} with $n=d+1$ to estimate the above summation. Then we see 
 \begin{equation}\label{MainForm} \mathcal{SQ}(A)+ \frac{\mathcal{ZR}(A)}{2}= \frac{1}{2q} q^{2(d+1)} \sum_{\widetilde{m}\in \mathbb F_q^{d+1}} \widehat{C_{d+1}}(\widetilde{m}) |\widehat{E}(\widetilde{m})|^2.\end{equation}
We replace the above value $\widehat{C_{d+1}}(\widetilde{m})$ by  \eqref{FFC} of Lemma \ref{FTForm}. Then a direct computation gives us the following:
$$\mathcal{SQ}(A)+ \frac{\mathcal{ZR}(A)}{2}=\frac{|E|^2}{2q^2} +\frac{q^{d-1} \eta(-1) G_1^{d+1}}{2} \sum_{\widetilde{m}\in \mathbb F_q^{d+1}}  \sum_{s\ne 0} \eta^{d+1}{(s)} \chi\left(\frac{||\widetilde{m}||_{C_{d+1}}}{-4s}\right)  
|\widehat{E}(\widetilde{m})|^2.$$
Notice that $|E|=|A|q$ and $\widehat{\mathbb F_q}(t)=1$ for $t=0,$ and $0$ otherwise. For $\widetilde{m}=(m,t)\in \mathbb F_q^d\times \mathbb F_q,$ we can check that
$$\widehat{E}(\widetilde{m})=\widehat{A\times \mathbb F_q} (m, t)=\widehat{A}(m) \widehat{\mathbb F_q}(t).$$
Therefore, we obtain
$$\mathcal{SQ}(A)+ \frac{\mathcal{ZR}(A)}{2}=\frac{|A|^2}{2} +\frac{q^{d-1} \eta(-1) G_1^{d+1}}{2} \sum_{m\in \mathbb F_q^{d}}  \sum_{s\ne 0} \eta^{d+1}{(s)} \chi\left(\frac{||m||}{-4s}\right)  
|\widehat{A}(m)|^2.$$
By a simple change of variables, $1/(-4s) \to s,$ and the properties of $\eta$,  
\begin{equation}\label{A+F} 
\mathcal{SQ}(A)+ \frac{\mathcal{ZR}(A)}{2}=\frac{|A|^2}{2} +\frac{q^{d-1} \eta^d(-1) G_1^{d+1}}{2} \sum_{m\in \mathbb F_q^{d}}  \sum_{s\ne 0} \eta^{d+1}{(s)} \chi\left({s||m||}\right)|\widehat{A}(m)|^2.  
\end{equation}
For better exposition, we will use the following notations.
\begin{definition} Let $A\subset \mathbb F_q^d.$ We define that
$$\Omega^+(A)=\sum_{m\in \mathbb F_q^d: \eta(||m||)=1} |\widehat{A}(m)|^2,$$ 
$$\Omega^-(A)=\sum_{m\in \mathbb F_q^d: \eta(||m||)=-1} |\widehat{A}(m)|^2,$$ 
and
$$\Omega^0(A)=\sum_{m\in \mathbb F_q^d: ||m||=0} |\widehat{A}(m)|^2.$$
\end{definition}
It is clear that $\Omega^+(A), \Omega^-(A)$, and $\Omega^0(A)$ are  non-negative real numbers. In addition, notice that if $A$ is a subset of $\mathbb F_q^d,$ then
\begin{equation}\label{triv} \Omega^0 (A)\ge |\widehat{A}(0,\ldots, 0)|^2 = q^{-2d} |A|^2.\end{equation}

 Recall that the sphere $S_0$ in $\mathbb F_q^d$ with zero radius, centered at the origin, is defined by 
$$ S_0:=\{x\in \mathbb F_q^d: ||x||=0\}.$$ With this notation and the definition of the Fourier transform, we have 
 $$\Omega^0(A)=\sum_{m\in S_0}|\widehat{A}(m)|^2=q^{-d}\sum_{x,y\in A} \widehat{S_0}(x-y).$$ 
%We will also use the Fourier transform on the sphere $S_0$ in $\mathbb F_q^d$ with zero ratius. 
By the same argument as in the proof of Lemma \ref{FTForm}, the Fourier transform of $S_0$, denoted by $\widehat{S_0},$ is given as follows (see, for example, \cite{IK10}): for $m\in \mathbb F_q^d,$ 
\begin{equation}\label{S0Fourier} \widehat{S_0}(m)=\frac{\delta_0(m)}{q} + q^{-d-1} \eta^d(-1) G_1^d \sum_{s\ne 0} \eta^d(s) \chi\left(\frac{||m||}{4s}\right).\end{equation}
In particular,  when $d\ge 3$ is odd, one has
$$ \Omega^0(A)=q^{-d-1}|A| +q^{-2d-1}\eta(-1)G_1^{d+1} \sum_{x,y\in A} \eta(||x-y||).$$
Since $\Omega^0(A)$ is a real number and $|G_1|=q^{1/2},$ we see that if $d\ge 3$ is odd, then
$$\Omega^0(A)\le q^{-d-1}|A| +q^{-\frac{3d+1}{2}} |A|^2.$$
One the other hand,  the Plancherel theorem gives the following estimate:
$$ \Omega^0(A) \le \sum_{m\in \mathbb F_q^d} |\widehat{A}(m)|^2 =q^{-d}|A|.$$
In conclusion, for any odd $d\ge 3$, we get
\begin{equation}\label{Omezero}
\Omega^0(A) \le \min\{q^{-d}|A|, ~q^{-d-1}|A| +q^{-\frac{3d+1}{2}} |A|^2\}.
\end{equation}
In odd dimensions $d\ge 3,$ we obtain the following formula.
\begin{proposition}\label{pro4.2} Let $A$ be a set in $\mathbb F_q^d.$
\begin{enumerate}
\item  If  $d\equiv 3 \mod{4}$ and $ q\equiv 3 \mod{4},$ 
then
$$\mathcal{SQ}(A)+ \frac{\mathcal{ZR}(A)}{2}=\frac{|A|^2}{2} -\frac{q^{\frac{3d+1}{2}}}{2}\Omega^0(A)+ \frac{q^{\frac{d-1}{2}}|A|}{2}.$$
\item If $d\equiv 1 \mod{4},$ or $d\equiv 3 \mod{4}$ and $q\equiv 1 \mod{4},$ then
$$\mathcal{SQ}(A)+ \frac{\mathcal{ZR}(A)}{2}=\frac{|A|^2}{2} +\frac{q^{\frac{3d+1}{2}}}{2}\Omega^0(A)- \frac{q^{\frac{d-1}{2}}|A|}{2}.$$
\end{enumerate}
\end{proposition}
\begin{proof} Since $d\ge 3$ is odd,  the inequality  \eqref{A+F} becomes
$$ \mathcal{SQ}(A)+ \frac{\mathcal{ZR}(A)}{2}=\frac{|A|^2}{2} +\frac{q^{d-1} \eta(-1) G_1^{d+1}}{2} \sum_{m\in \mathbb F_q^{d}}  \sum_{s\ne 0} \chi\left({s||m||}\right)|\widehat{A}(m)|^2. $$ 
By the orthogonality of $\chi$, we compute the sum over $s\ne 0$. Then we get
$$\mathcal{SQ}(A)+ \frac{\mathcal{ZR}(A)}{2}=\frac{|A|^2}{2} 
+\frac{q^{d} \eta(-1) G_1^{d+1}}{2} \sum_{m\in  S_0} |\widehat{A}(m)|^2 -\frac{q^{d-1} \eta(-1) G_1^{d+1}}{2} \sum_{m\in \mathbb F_q^d} |\widehat{A}(m)|^2.$$
Since $\sum\limits_{m\in \mathbb F_q^d} |\widehat{A}(m)|^2=q^{-d}|A|$, we have 
\begin{equation}
\mathcal{SQ}(A)+ \frac{\mathcal{ZR}(A)}{2}=\frac{|A|^2}{2} 
+\frac{q^{d} \eta(-1) G_1^{d+1}}{2} \Omega^0(A) -\frac{\eta(-1) G_1^{d+1}|A|}{2q}. 
\end{equation}
To prove the first part (1) of the proposition, it suffices to show that $\eta(-1)G_1^{d+1}= -q^{\frac{d+1}{2}}.$ Since $d\equiv 3 \mod{4}$ and $ q\equiv 3 \mod{4},$ the equation follows from Corollary \ref{Corm} (1) with $n=d+1.$ 
To prove the second part (2) of the proposition, it is enough to show that $\eta(-1)G_1^{d+1}= q^{\frac{d+1}{2}}.$ Notice that this equation  follows immediately from Corolalry \ref{Corm} (2) with $n=d+1.$
\end{proof}

In even dimensions $d\ge 2,$ we have
\begin{proposition} \label{pro4.3}
Let $A$ be a set in $\mathbb F_q^d.$
\begin{enumerate}
\item If $d\ge 2$ is even, then 
$$\mathcal{SQ}(A)+ \frac{\mathcal{ZR}(A)}{2}=\frac{|A|^2}{2} +\frac{q^{d-1}  G_1^{d+2}}{2} \Omega^+(A)-\frac{q^{d-1}  G_1^{d+2}}{2} \Omega^-(A).$$
\item In particular, if  $d\equiv 2 \mod{4},$ then
$$\mathcal{SQ}(A)+ \frac{\mathcal{ZR}(A)}{2}=\frac{|A|^2}{2} +\frac{q^{\frac{3d}{2}}}{2}\Omega^+(A)-\frac{q^{\frac{3d}{2}}}{2}\Omega^-(A). $$
%\item On the other hand, if $d\equiv 0 \mod{4},$ then
%$$\mathcal{SQ}(A)+ \frac{\mathcal{ZR}(A)}{2}=\frac{|A|^2}{2} +\frac{\eta(-1)q^{\frac{3d}{2}}}{2}\Omega^+(A)-\frac{\eta(-1)q^{\frac{3d}{2}}}{2}\Omega^-(A).$$
\end{enumerate}
\end{proposition}
\begin{proof}
Since $d\ge 2$ even,  \eqref{A+F} implies that
$$\mathcal{SQ}(A)+ \frac{\mathcal{ZR}(A)}{2}=\frac{|A|^2}{2} +\frac{q^{d-1}  G_1^{d+1}}{2} \sum_{m\in \mathbb F_q^{d}}  \sum_{s\ne 0} \eta{(s)} \chi\left({s||m||}\right)|\widehat{A}(m)|^2. $$ 
Note that the sum over $s\ne 0$ is $G_1 \eta(||m||).$ Since $\eta(||m||)=0$ for $||m||=0,$ the first statement (1) of the proposition follows. Proposition \ref{pro4.3} (2)  is a  direct consequence of Proposition \ref{pro4.3} (1) since  $G_1^{d+2}=q^{(d+2)/2}$ for $d\equiv 2 \mod{4}.$ %and $G_1^{d+2}=\eta(-1)q G_1^{d}= \eta(-1)q^{(d+2)/2}$ for $d\equiv 0 \mod{4}.$
\end{proof}

\subsection{$\mathcal{ZR}(A)/2$ value} 
Recall that for $A\subset \mathbb F_q^d,$ we have
$$\mathcal{ZR}(A)=\sum_{x,y\in A: ||x-y||=0} 1.$$ 
By applying Lemma \ref{AlexF} with  $n=d$ and $V=S_0,$ we have
$$ \frac{\mathcal{ZR}(A)}{2}= \frac{q^{2d}}{2} \sum_{m\in \mathbb F_q^d} \widehat{S_0}(m) |\widehat{A}(m)|^2.$$
The following equality  can be easily obtained by a direct computation after replacing the above value $\widehat{S_0}(m)$ by the value in \eqref{S0Fourier}.
\begin{equation}\label{defA0}\frac{\mathcal{ZR}(A)}{2}= \frac{|A|^2}{2q}+ \frac{q^{d-1} \eta^d(-1) G_1^d }{2} \sum_{m\in \mathbb F_q^d} 
\left(\sum_{s\ne 0} \eta^d(s) \chi\left(s||m||\right)  \right) |\widehat{A}(m)|^2,\end{equation}
where we also used a simple change of variables, $1/(4s)\to s,$ and the properties of the quadratic character $\eta$ of $\mathbb F_q.$\\

In odd dimensions $d\ge 3,$  the value $\mathcal{ZR}(A)/2$ is written as follows.

\begin{proposition}\label{pro4.4} Let $A$ be a set in $\mathbb F_q^d.$
\begin{enumerate}
\item  If  $d\equiv 3 \mod{4}$ and $ q\equiv 3 \mod{4},$ 
then
$$ \frac{\mathcal{ZR}(A)}{2}=\frac{|A|^2}{2q} -\frac{q^{\frac{3d-1}{2}}}{2}\Omega^+(A)+\frac{q^{\frac{3d-1}{2}}}{2}\Omega^-(A). $$
\item If $d\equiv 1 \mod{4},$ or $d\equiv 3 \mod{4}$ and $q\equiv 1 \mod{4},$ then
$$\frac{\mathcal{ZR}(A)}{2}=\frac{|A|^2}{2q} +\frac{q^{\frac{3d-1}{2}}}{2}\Omega^+(A)-\frac{q^{\frac{3d-1}{2}}}{2}\Omega^-(A).$$
\end{enumerate}
\end{proposition}
\begin{proof} Since $d$ is odd,  $\eta^d=\eta$ and so 
the equality \eqref{defA0} becomes
\begin{align*} \frac{\mathcal{ZR}(A)}{2}&= \frac{|A|^2}{2q}+ \frac{q^{d-1} \eta(-1) G_1^d }{2} \sum_{m\in \mathbb F_q^d} 
\left(\sum_{s\ne 0} \eta(s) \chi\left(s||m||\right)  \right) |\widehat{A}(m)|^2\\
&=\frac{|A|^2}{2q}+ \frac{q^{d-1} \eta(-1) G_1^{d+1} }{2} \sum_{m\in \mathbb F_q^d} 
 \eta(||m||) |\widehat{A}(m)|^2.\end{align*}
Since $\eta{(||m||)}=0$ for $||m||=0,$  it follows by the definitions of $\Omega^+(A)$ and $\Omega^-(A)$ that 
$$ \frac{\mathcal{ZR}(A)}{2} = \frac{|A|^2}{2q}+ \frac{q^{d-1} \eta(-1) G_1^{d+1} }{2} \Omega^+(A) - \frac{q^{d-1} \eta(-1) G_1^{d+1} }{2} \Omega^-(A).$$
Applying  Corollary \ref{Corm} (1) and (2) with $n=d+1$,  we complete the proof.
\end{proof}

In even dimensions $d\ge 2,$ we have
\begin{proposition} \label{pro4.5}
Let $A$ be a set in $\mathbb F_q^d.$
\begin{enumerate}
\item  If  $d\equiv 2 \mod{4}$ and $ q\equiv 3 \mod{4},$ then
$$ \frac{\mathcal{ZR}(A)}{2}= \frac{|A|^2}{2q} -\frac{q^{\frac{3d}{2}}}{2} 
\Omega^0(A)+ \frac{q^{\frac{d-2}{2}}|A|}{2}.$$

\item If $d\equiv 0 \mod{4},$ or $d\equiv 2 \mod{4}$ and $q\equiv 1 \mod{4},$ then
$$\frac{\mathcal{ZR}(A)}{2}=\frac{|A|^2}{2q} +\frac{q^{\frac{3d}{2}}}{2}\Omega^0(A)- \frac{q^{\frac{d-2}{2}}|A|}{2}.$$
\end{enumerate}
\end{proposition}
\begin{proof}
Since $d$ is even, $\eta^d\equiv 1.$ Hence, the equality \eqref{defA0} and the orthogonality of $\chi$ yield 
\begin{align*}\frac{\mathcal{ZR}(A)}{2}&= \frac{|A|^2}{2q}+ \frac{q^{d-1}  G_1^d }{2} \sum_{m\in \mathbb F_q^d} 
\left(\sum_{s\ne 0}  \chi\left(s||m||\right)  \right) |\widehat{A}(m)|^2\\
&=\frac{|A|^2}{2q}+ \frac{q^{d}  G_1^d }{2}\Omega^0(A) -\frac{q^{d-1}  G_1^d }{2} \sum_{m\in \mathbb F_q^d} |\widehat{A}(m)|^2.\end{align*} 
Notice by the Plancherel theorem that $\sum\limits_{m\in \mathbb F_q^d} |\widehat{A}(m)|^2=q^{-d}|A|.$ Then the proposition follows from the Gauss sum values given in  Corollary \ref{Corm} (3) and (4) with $n=d.$
\end{proof}

\section{Proofs of results on $\mathcal{SQ}(A)+\mathcal{ZR}(A)$ (Theorem \ref{mainN0})}
To get a good upper bound of $\mathcal{SQ}(A)+\mathcal{ZR}(A)$,  we will use the previously established estimates of $\mathcal{SQ}(A)+ \mathcal{ZR}(A)/2$ and $\mathcal{ZR}(A).$ In other words,  we will find an upper bound of the sum of $(\mathcal{SQ}(A)+\mathcal{ZR}(A)/2)$ and $\mathcal{ZR}(A)/2.$
%%%%%%%%%%%%%%%%%%%%%%%%%%%%%%%%%%%%%%%%%%%%%%%
 
 \subsection{ Proof of Theorem \ref{mainN0} (1)}
%We can write $\mathcal{SQ}(A)= (\mathcal{SQ}(A)+ \mathcal{ZR}(A)/2) +\mathcal{ZR}(A)/2.$  
 Since we assume that $d\equiv 3 \mod{4}$ and $ q\equiv 3 \mod{4},$ Propositions \ref{pro4.2} (1) and \ref{pro4.4} (1) play a key role in proving this theorem. \\
By inequality \eqref{triv} and Proposition \ref{pro4.2} (1), we see that
\begin{equation}\label{TK1}
(\mathcal{SQ}(A)+ \mathcal{ZR}(A)/2)\le \frac{|A|^2}{2} -\frac{|A|^2}{2q^{\frac{d-1}{2}}}+ \frac{q^{\frac{d-1}{2}}|A|}{2}.
\end{equation}  
Ignoring a negative term, Proposition \ref{pro4.4} (1) implies that 
$$\frac{\mathcal{ZR}(A)}{2}\le \frac{|A|^2}{2q} +\frac{q^{\frac{3d-1}{2}}}{2}\Omega^-(A). $$
One can check that
$$\Omega^{-}(A)\le \sum_{m\in \mathbb F_q^d} |\widehat{A}(m)|^2 -\Omega^0(A) \le q^{-d}|A|-|\widehat{A}(0,\ldots, 0)|^2=q^{-d}|A|-q^{-2d}|A|^2.$$ 
Thus, we obtain
$$\frac{\mathcal{ZR}(A)}{2}\le \frac{|A|^2}{2q} +\frac{q^{\frac{d-1}{2}}|A|}{2} - \frac{|A|^2}{2q^{\frac{d+1}{2}}}.$$

Since $\mathcal{SQ}(A)+ \mathcal{ZR}(A) =(\mathcal{SQ}(A)+ \mathcal{ZR}(A)/2)+ {\mathcal{ZR}(A)}/{2},$ we complete the proof by adding the above two inequalities. \qed\\

\subsection{Proof of Theorem \ref{mainN0} (2)}
Since $d\equiv 1 \mod{4},$ or $d\equiv 3 \mod{4}$ and $q\equiv 1 \mod{4},$ we are able to use the following result of Proposition \ref{pro4.2} (2):
$$\mathcal{SQ}(A)+ \frac{\mathcal{ZR}(A)}{2}=\frac{|A|^2}{2} +\frac{q^{\frac{3d+1}{2}}}{2}\Omega^0(A)- \frac{q^{\frac{d-1}{2}}|A|}{2}.$$
Since $\Omega^0(A)$ is dominated by $\sum\limits_{m\in \mathbb F_q^d} |\widehat{A}(m)|^2 - \Omega^+(A)=q^{-d}|A|-\Omega^+(A),$  it follows that
$$\mathcal{SQ}(A)+ \frac{\mathcal{ZR}(A)}{2}\le \frac{|A|^2}{2} +\frac{q^{\frac{3d+1}{2}}}{2}(q^{-d}|A|-\Omega^+(A))- \frac{q^{\frac{d-1}{2}}|A|}{2}. $$

Proposition \ref{pro4.4} (2) implies that
$$\frac{\mathcal{ZR}(A)}{2}\le \frac{|A|^2}{2q} +\frac{q^{\frac{3d-1}{2}}}{2}\Omega^+(A).$$
Adding the above two estimates, we have
 $$\mathcal{SQ}(A)+\mathcal{ZR}(A) \le \frac{|A|^2}{2} +\frac{|A|^2}{2q} + \frac{q^{\frac{d+1}{2}} |A|}{2} -\frac{q^{\frac{d-1}{2}} |A|}{2}-\left( \frac{q^{\frac{3d+1}{2}}}{2}- \frac{q^{\frac{3d-1}{2}}}{2}\right)\Omega^+(A).$$
Since the term containing $\Omega^+(A)$ value is negative or zero, the theorem follows. \qed\\
\subsection{Proof of Theorem \ref{mainN0} (3)}
 Since $d\equiv 2 \mod{4}$ and $ q\equiv 3 \mod{4},$  the proof can use Propositions \ref{pro4.3} (2) and \ref{pro4.5} (1). \\
 Since $-\Omega^-(A)\le 0,$ Proposition \ref{pro4.3} (2) implies that
 \begin{equation}\label{ThangDoowon3}\mathcal{SQ}(A)+ \frac{\mathcal{ZR}(A)}{2}\le \frac{|A|^2}{2} +\frac{q^{\frac{3d}{2}}}{2}\Omega^+(A). \end{equation} 
 Combining this estimate with the fact that $\Omega^+(A) \le q^{-d}|A|-\Omega^0(A)$, we obtain
 $$\mathcal{SQ}(A)+ \frac{\mathcal{ZR}(A)}{2}\le \frac{|A|^2}{2} +\frac{q^{\frac{3d}{2}}}{2} (q^{-d}|A|-\Omega^0(A)).$$
By Proposition \ref{pro4.5} (1), we have
 $$\frac{\mathcal{ZR}(A)}{2}= \frac{|A|^2}{2q} -\frac{q^{\frac{3d}{2}}}{2} 
\Omega^0(A)+ \frac{q^{\frac{d-2}{2}}|A|}{2}.$$
By adding the above two inequalities, we have
 $$\mathcal{SQ}(A)+\mathcal{ZR}(A) \le \frac{|A|^2}{2} +\frac{|A|^2}{2q}-q^{\frac{3d}{2}} 
\Omega^0(A)+ \frac{q^{\frac{d}{2}} |A|}{2} +\frac{q^{\frac{d-2}{2}}|A|}{2}.$$ 
Since $-\Omega^0(A)\le -|\widehat{A}(0,\ldots,0)|^2=-q^{-2d}|A|^2,$  we obtain the statement of the theorem. \qed\\
\subsection{Proof of Theorem \ref{mainN0} (4)}
Since $d\equiv 0 \mod{4},$ or $d\equiv 2 \mod{4}$ and $q\equiv 1 \mod{4},$ we may invoke Proposition \ref{pro4.3} (1) and \ref{pro4.5} (2). \\ 
First, Proposition \ref{pro4.3} (1) tells us that for even integer $d\ge 2,$
$$\mathcal{SQ}(A)+ \frac{\mathcal{ZR}(A)}{2}=\frac{|A|^2}{2} +\frac{q^{d-1}  G_1^{d+2}}{2} \Omega^+(A)-\frac{q^{d-1}  G_1^{d+2}}{2} \Omega^-(A):=\frac{|A|^2}{2}+ I + II.$$
Since $d$ is  even,  Lemma \ref{ExplicitGauss} implies that $G_1^{d+2}=\pm q^{(d+2)/2}$ which is a real number. Moreover, both $ \Omega^+(A)$ and $\Omega^-(A)$ are non-negative real numbers. Hence, 
one of  $I$ and $II$ is exactly non-negative and the other is non-positive. Therefore, in the case when  $G_1^{d+2}= q^{(d+2)/2},$ we have 
\begin{equation}\label{asseq1} \mathcal{SQ}(A)+ \frac{\mathcal{ZR}(A)}{2}\le \frac{|A|^2}{2}+ I=\frac{|A|^2}{2}+\frac{q^{\frac{3d}{2}}}{2} \Omega^+(A).\end{equation}
On the other hand, when  $G_1^{d+2}= -q^{(d+2)/2},$ we have  
\begin{equation}\label{asseq2}\mathcal{SQ}(A)+ \frac{\mathcal{ZR}(A)}{2}\le\frac{|A|^2}{2}+ II=\frac{|A|^2}{2}+\frac{q^{\frac{3d}{2}}}{2} \Omega^-(A).\end{equation}
In any case, we have
$$\mathcal{SQ}(A)+ \frac{\mathcal{ZR}(A)}{2}\le\frac{|A|^2}{2}+\frac{q^{\frac{3d}{2}}}{2} (q^{-d} |A|-\Omega^0(A)),$$
since $\max\{\Omega^+(A), ~\Omega^-(A)\} \le  \sum\limits_{m\in \mathbb F_q^d} |\widehat{A}(m)|^2 -\Omega^0(A) = q^{-d}|A|-\Omega^0(A).$\\
By Proposition \ref{pro4.5} (2), we have
$$\frac{\mathcal{ZR}(A)}{2}=\frac{|A|^2}{2q} +\frac{q^{\frac{3d}{2}}}{2}\Omega^0(A)- \frac{q^{\frac{d-2}{2}}|A|}{2}.$$
Hence, adding the above two estimates, we obtain the required upper bound of $\mathcal{SQ}(A)+ \mathcal{ZR}(A).$ \qed\\
\section{Proofs of results on $\mathcal{SQ}(A)$ in odd dimensions (Theorem \ref{mainNOdd})}

\subsection{Proof of Theorem \ref{mainNOdd} (1)}
Since the assumtion that $d\equiv 3 \mod{4}$ and $ q\equiv 3 \mod{4}$ is the same as that of Theorem
\ref{mainN0} (1), we will make use certain estimates given in the proof of Theorem \ref{mainN0} (1).\\
As in \eqref{TK1}, we have
\begin{equation}\label{TK11}
    (\mathcal{SQ}(A)+ \mathcal{ZR}(A)/2)\le \frac{|A|^2}{2} -\frac{|A|^2}{2q^{\frac{d-1}{2}}}+ \frac{q^{\frac{d-1}{2}}|A|}{2}.
\end{equation}
Since $\Omega^-(A)$ is non-negative, Proposition \ref{pro4.4} (1) implies that
$$ -\mathcal{ZR}(A)/2\le  -\frac{|A|^2}{2q} +\frac{q^{\frac{3d-1}{2}}}{2} \Omega^+(A).  $$

As before, we can check that $\Omega^+(A)$ is bounded by   $q^{-d}|A|-q^{-2d} |A|^2$ . Thus, we have
\begin{equation} \label{TK2}
-\mathcal{ZR}(A)/2\le  -\frac{|A|^2}{2q} +\frac{q^{\frac{3d-1}{2}}}{2}(q^{-d}|A|-q^{-2d} |A|^2)= -\frac{|A|^2}{2q}+\frac{q^{\frac{d-1}{2}}|A|}{2}-\frac{|A|^2}{2q^{\frac{d+1}{2}}}.
\end{equation}

By the definition of $\mathcal{ZR}(A),$ it is clear that $\mathcal{ZR}(A)\ge |A|.$ Hence, we also have
\begin{equation}\label{TK3} -\mathcal{ZR}(A)/2\le -|A|/2.\end{equation}

Since we can write $\mathcal{SQ}(A)= (\mathcal{SQ}(A)+ \mathcal{ZR}(A)/2) -\mathcal{ZR}(A)/2,$ the estimates \eqref{TK11} and \eqref{TK2} imply that
$$\mathcal{SQ}(A)\le \frac{|A|^2}{2} + q^{\frac{d-1}{2}} |A| -\frac{|A|^2}{2q}- \frac{|A|^2}{2 q^{\frac{d-1}{2}}} - \frac{|A|^2}{2 q^{\frac{d+1}{2}}}.$$

Combining \eqref{TK11} and \eqref{TK3}, we also obtain that
$$ \mathcal{SQ}(A) \le \frac{|A|^2}{2} + \frac{q^{\frac{d-1}{2}} |A|}{2}
  - \frac{|A|^2}{2 q^{\frac{d-1}{2}}} -\frac{|A|}{2}.$$
  
Comparing the above two estimates of $\mathcal{SQ}(A),$ we obtain the statement of Theorem \ref{mainNOdd} (1).
\qed \\

\subsection{Proof of Theorem \ref{mainNOdd} (2)}
Since  $d\equiv 1 \mod{4},$ or $d\equiv 3 \mod{4}$ and $q\equiv 1 \mod{4},$  we will make use of Propositions \ref{pro4.2} (2) and \ref{pro4.4} (2).\\
Combining inequality \eqref{Omezero} and Proposition \ref{pro4.2} (2), we see that
$$\mathcal{SQ}(A)+ \frac{\mathcal{ZR}(A)}{2}\le\frac{|A|^2}{2} +\frac{q^{\frac{3d+1}{2}}}{2}\min\{q^{-d}|A|, ~q^{-d-1}|A| +q^{-\frac{3d+1}{2}} |A|^2\}- \frac{q^{\frac{d-1}{2}}|A|}{2}.$$
 As seen in \eqref{TK3},  we have
 $$  -\mathcal{ZR}(A)/2\le -|A|/2.$$
 Hence, the above two estimates imply that
\begin{equation}\label{OE1} 
\mathcal{SQ}(A) \le \frac{|A|^2}{2} - \frac{q^{\frac{d-1}{2}} |A|}{2}
  -\frac{|A|}{2}+ \min\left\{ \frac{q^{\frac{d+1}{2}} |A|}{2},~ \frac{q^{\frac{d-1}{2}} |A|}{2}+ \frac{|A|^2}{2} \right\}.
\end{equation} 
Next, we deduce another upper bound of $\mathcal{SQ}(A).$ 
Since $\Omega^0(A)=\sum\limits_{m\in \mathbb F_q^d} |\widehat{A}(m)|^2 -\Omega^+(A)-\Omega^-(A) =q^{-d}|A|-\Omega^+(A)-\Omega^-(A)\le q^{-d}|A|-\Omega^-(A),$ it follows by Proposition \ref{pro4.2} (2) that
\begin{equation}\label{Doowon2}\mathcal{SQ}(A)+ \frac{\mathcal{ZR}(A)}{2}\le \frac{|A|^2}{2} +\frac{q^{\frac{3d+1}{2}}}{2}(q^{-d}|A|-\Omega^-(A))- \frac{q^{\frac{d-1}{2}}|A|}{2}.\end{equation}
Since $-\Omega^+(A)\le 0,$  Proposition \ref{pro4.4} (2) implies that 
$$-\frac{\mathcal{ZR}(A)}{2}\le -\frac{|A|^2}{2q} +\frac{q^{\frac{3d-1}{2}}}{2}\Omega^-(A).$$
Since $-\Omega^-(A)\le 0,$ adding the above two estimates gives us 
$$\mathcal{SQ}(A)\le \frac{|A|^2}{2} + \frac{q^{\frac{d+1}{2}} |A|}{2} -\frac{q^{\frac{d-1}{2}} |A|}{2}-\frac{|A|^2}{2q}.$$
Combining this estimate with inequality \eqref{OE1}, we complete the proof of Theorem \ref{mainNOdd} (2).   
\qed
\\
\section{Proofs of results on $\mathcal{SQ}(A)$ in even dimensions (Theorem \ref{mainNEven})}
We begin by proving an upper bound of $\mathcal{SQ}(A)$ for even $d\ge 2.$ We will use some certain estimates appearing in the proof of Theorem \ref{mainN0}(4).
\begin{lemma}\label{evenA+}
Let $A$ be a subset of $\mathbb F_q^d.$ If $d\ge 2 $ is even, then we have
$$ \mathcal{SQ}(A)\le \frac{|A|^2}{2} + \frac{q^{\frac{d}{2}}  }{2}|A|- \frac{|A|^2}{2q^{\frac{d}{2}}  }-\frac{ |A| }{2}.$$
\end{lemma}
\begin{proof}
As observed in the proof of Theorem \ref{mainN0} (4), either the inequality \eqref{asseq1} or the inequality \eqref{asseq2} happens. Without loss of generality, we can assume that  the inequality \eqref{asseq1} holds. By the same method, we can deal with the other case.  As seen before, we have
$\Omega^+(A)\le \sum_{m\in \mathbb F_q^d} |\widehat{A}(m)|^2 -\Omega^0(A)\le q^{-d}|A|-q^{-2d}|A|^2$.
Thus, we obtain by \eqref{asseq1} that for even $d\ge 2,$
\begin{equation}\label{evenEasy}\mathcal{SQ}(A)+ \frac{\mathcal{ZR}(A)}{2} \le \frac{|A|^2}{2}+\frac{q^{\frac{3d}{2}}}{2} ( q^{-d}|A|-q^{-2d}|A|^2 )=\frac{|A|^2}{2} + \frac{q^{\frac{d}{2}}  }{2}|A|- \frac{|A|^2}{2q^{\frac{d}{2}}}.\end{equation}
Since $\mathcal{ZR}(A)\ge|A|$ by the definition of $\mathcal{ZR}(A),$ we have
$$ -\frac{\mathcal{ZR}(A)}{2} \le -\frac{|A|}{2}.$$
 The lemma follows by adding the above two estimates. 
\end{proof}

\subsection{Proof of Theorem \ref{mainNEven} (1)}
Since  $d\equiv 2 \mod{4}$ and $ q\equiv 3 \mod{4},$ we can also use certain estimates which were already established in the proof of Theorem \ref{mainN0} (3). As seen in \eqref{ThangDoowon3}, we have
$$\mathcal{SQ}(A)+ \frac{\mathcal{ZR}(A)}{2}\le \frac{|A|^2}{2} +\frac{q^{\frac{3d}{2}}}{2}\Omega^+(A).$$
 Using the fact that $\Omega^+(A)\le q^{-d}|A|-\Omega^0{A},$  we obtain
$$\mathcal{SQ}(A)+ \frac{\mathcal{ZR}(A)}{2}\le \frac{|A|^2}{2}+ \frac{q^{\frac{3d}{2}}}{2}(q^{-d}|A|-\Omega^0{A})$$
By Proposition \ref{pro4.5} (1), it follows that 
 $$ -\frac{\mathcal{ZR}(A)}{2}= -\frac{|A|^2}{2q} +\frac{q^{\frac{3d}{2}}}{2} 
\Omega^0(A)- \frac{q^{\frac{d-2}{2}}|A|}{2}.$$

Adding the above two estimates gives the statement of the theorem. \qed\\

%$$\mathcal{SQ}(A)\le \frac{|A|^2}{2}+\frac{q^{\frac{d}{2}}|A|}{2}- \frac{|A|^2}{2q}- \frac{q^{\frac{d-2}{2}}|A|}{2}.$$ This completes the proof.
%By Lemma \ref{evenA+}, we also have
%$$ \mathcal{SQ}(A)\le \frac{|A|^2}{2} + \frac{q^{\frac{d}{2}}|A|  }{2}- \frac{|A|^2}{2q^{\frac{d}{2}}  }-\frac{ |A| }{2}.$$
\subsection{ Proof of Theorem \ref{mainNEven} (2)}
Since $d\equiv 0 \mod{4},$ or $d\equiv 2 \mod{4}$ and $q\equiv 1 \mod{4},$ we will use the estimate \eqref{evenEasy} and Proposition \ref{pro4.5} (2).\\
Since $d$ is even, it follows by \eqref{evenEasy} that
 \begin{equation}\label{ThangKoh4}\mathcal{SQ}(A)+ \frac{\mathcal{ZR}(A)}{2} \le \frac{|A|^2}{2} + \frac{q^{\frac{d}{2}}  }{2}|A|- \frac{|A|^2}{2q^{\frac{d}{2}}}.\end{equation}

Since $-\Omega^0(A) \le -|\widehat{A}(0,\ldots, 0)|^2 = -q^{-2d}|A|^2,$ Proposition \ref{pro4.5} (2) implies that
$$-\frac{\mathcal{ZR}(A)}{2}\le -\frac{|A|^2}{2q} -\frac{|A|^2}{2q^{\frac{d}{2}}}+ \frac{q^{\frac{d-2}{2}}|A|}{2}.$$
Thus, adding the above two estimates gives us an upper bound of $\mathcal{SQ}(A)$ as follows: 
$$  \mathcal{SQ}(A)\le \frac{|A|^2}{2} + \frac{q^{\frac{d}{2}} |A|}{2}-\frac{|A|^2}{q^{\frac{d}{2}}} -\frac{|A|^2}{2q}+ \frac{q^{\frac{d-2}{2}}|A|}{2}.$$ 

Lemma \ref{evenA+} also gives the following upper bound of $\mathcal{SQ}(A)$:
$$ \mathcal{SQ}(A)\le \frac{|A|^2}{2} + \frac{q^{\frac{d}{2}}  }{2}|A|- \frac{|A|^2}{2q^{\frac{d}{2}}  }-\frac{ |A| }{2}.$$
The statement of the theorem follows by a direct comparison with two upper bounds of $\mathcal{SQ}(A).$ \qed\\

\section*{Acknowledgments}
Doowon Koh was supported by the National Research Foundation of Korea (NRF) grant funded by the Korea government (MIST) (No. NRF-2018R1D1A1B07044469). T. Pham would like to thank to the VIASM for the hospitality and for the excellent working condition.

\end{document}